\documentclass[11pt]{amsart}

\usepackage{eurosym,graphicx,amscd,color,amsmath,amsfonts,amssymb,amsthm,centernot,geometry,hyperref,mathtools,tikz}
\usepackage[initials]{amsrefs}
\usepackage[norefs,nocites]{refcheck}

\usetikzlibrary{arrows.meta, positioning, decorations.pathreplacing}

\newtheorem{theorem}{Theorem}[section]

\newtheorem{corollary}[theorem]{Corollary}
\newtheorem{definition}[theorem]{Definition}
\newtheorem{example}[theorem]{Example}
\newtheorem{lemma}[theorem]{Lemma}
\newtheorem{proposition}[theorem]{Proposition}
\newtheorem{remark}[theorem]{Remark}

\newtheorem*{SQP}{Separable Quotient Problem (SQP)}

\begin{document}
\title{The Separable Quotient Problem and the Lotz Property in Banach Spaces}

\author[Ara\'{u}jo]{Gustavo Ara\'{u}jo$^{*}$}
\address[G. Ara\'{u}jo]{\mbox{}\newline\indent Departamento de Matem\'{a}tica \newline\indent
Universidade Estadual da Para\'{i}ba \newline\indent
Campina Grande - PB \newline\indent
58.429-500 (Brazil) \newline\indent
and \newline\indent
Instituto de Matemática Interdisciplinar (IMI) \newline\indent
Departamento de Análisis Matemático y Matemática Aplicada \newline\indent
Facultad de Ciencias Matemáticas \newline\indent
Universidad Complutense de Madrid \newline\indent
Madrid \newline\indent
28040 (Spain)}
\email{gustavoaraujo@servidor.uepb.edu.br}

\author[Barbosa]{Anderson Barbosa}
\address[A. Barbosa]{\mbox{}\newline\indent Departamento de Matem\'atica \newline\indent
Universidade Federal da Para\'iba \newline\indent
Jo\~ao Pessoa - PB \newline\indent 
58.051-900 (Brazil) \newline\indent
and \newline\indent
Departamento de Ci\^encias Exatas e Aplicadas \newline\indent
Centro de Ensino Superior do Serid\'o \newline\indent
Universidade Federal do Rio Grande do Norte \newline\indent
Caicó - RN \newline\indent
59.300-000 (Brazil)
}
\email{afsb@academico.ufpb.br}

\author[Seoane]{Juan B. Seoane-Sepúlveda}
\address[J. B. Seoane-Sepúlveda]{\mbox{}\newline\indent Instituto de Matem\'atica Interdisciplinar (IMI) \newline\indent
Departamento de An\'alisis Matem\'atico y Matem\'atica Aplicada \newline\indent
Facultad de Ciencias Matem\'aticas \newline\indent
Universidad Complutense de Madrid \newline\indent
Madrid \newline\indent 
28040 (Spain)
}
\email{jseoane@ucm.es}

\thanks{$^{*}$Corresponding author.}
\thanks{This paper was written during a 1-year visit of the first author at the Universidad Complutense de Madrid. This study was financed in part by the Coordena\c{c}\~{a}o de Aperfei\c{c}oamento de Pessoal de N\'{\i}vel Superior - Brasil (CAPES) - Finance Code 001}

\keywords{Separable Quotient Problem, Operator range, Semigroup Theory.}

\subjclass{46B20, 46B42, 46B99}

\begin{abstract}
In this paper, in addition to establishing several new sufficient or equivalent conditions for the well-known Separable Quotient Problem (SQP), we show that every infinite-dimensional Banach space lacking the Lotz property admits a separable quotient, thereby revealing a surprising connection between the SQP and semigroup theory.
\end{abstract}

\maketitle

\tableofcontents

\section{Introduction and preliminaries}

The following problem, commonly referred to as the \textit{Separable Quotient Problem}, is one of the most significant open questions in the theory of Banach spaces:

\begin{SQP}
Given an infinite-dimensional Banach space $X$, does there exist a closed infinite-dimensional subspace $Y$ of $X$ such that $X/Y$ is an infinite-dimensional separable Banach space?
\end{SQP}

As pointed out in \cite{Brech2023}, this problem has likely been considered since the 1930s, alongside other fundamental questions that emerged from Banach's seminal work \cite{banachtese}. Although the formulation of the question is commonly attributed to Banach and Mazur, there is no explicit mention of it in Banach's book \cite{banach1,banach3}, nor is there any formal record of it. Furthermore, the earliest known explicit reference appears to be in Rosenthal's paper \cite{rosenthal}.

Quotients of Banach spaces serve as an essential tool for analyzing and understanding the structural properties of the spaces themselves. This is exemplified by the so-called three-space problems, which investigate the following: If two of the three spaces $X$, $Y \subseteq X$, and $X/Y$ possess a certain property, can we conclude that the third space also shares this property? We refer to \cite{litrico} for various examples of three-space properties in the literature.

During the 1960s and 1970s, several positive results concerning the SQP were obtained. The following classes of Banach spaces are known to admit separable quotients:

\begin{itemize}
\item Banach spaces containing $c_0$ (Bessaga--Pe\l czy\'nski \cite{bessaga1958b});
\item Reflexive spaces (Pe\l czy\'nski \cite{pelczynski1964});
\item Weakly compactly generated spaces (Amir--Lindenstrauss \cite{amir1968});
\item Dual spaces (Argyros--Dodos--Kanellopoulos \cite{argyros2008});
\item Spaces whose dual contains an unconditional basic sequence (Hagler--Johnson \cite{hagler1977});
\item Spaces whose itself or the topological dual contains a copy of $c_0$ or $\ell_1$ (Ferrando--Kakol--López-Pellicer--\'Sliwa \cite{CKLS});
\item Weakly countably determined spaces (W\'ojtowicz \cite{Wojtowicz1997});
\item Banach spaces with a Schauder decomposition (Saxon--Wilansky \cite{saxon1977}).
\end{itemize}

From now on, $\mathbb{K} = \mathbb{R}$ or $\mathbb{C}$. Here, as usual, $$\displaystyle c_0 = \{ (x_n)_{n\in\mathbb{N}} \in  \mathbb{K}^\mathbb{N} : \lim_{n \to \infty} x_n = 0 \},$$ and for $1\leq p< \infty$, $$\ell_p = \{ (x_n)_{n\in\mathbb{N}} \in \mathbb{K}^\mathbb{N} : \sum_{i=1}^{\infty}|x_i|^p < \infty \}$$ and $$\ell_\infty = \{ (x_n)_{n\in\mathbb{N}} \in \mathbb{K}^\mathbb{N} : \sup_{n\in \mathbb{N}} |x_n| < \infty \}.$$ A Banach space \( X \) is called \emph{weakly compactly generated} (WCG) if there exists a weakly compact subset \( K \subset X \) such that, in the strong (norm) topology, the subspace generated by $K$ is dense in $X$, that is, every element of $X$ can be approximated in norm by finite linear combinations of elements of $K$.

The Banach space $X$ is called \textit{weakly countably determined} (WCD) whenever there is a sequence $\mathcal{A}=(A_n)_{n\in\mathbb{N}}$ of absolutely convex, closed and bounded subsets of $X$ with $X=\bigcup_{I\in J(\mathcal{A})}\bigcap_{n\in I}A_n$ where $J(\mathcal{A})=\{I\subset \mathbb{N}:\ \bigcap_{n\in I}A_n\ \text{is weakly compact}\}$. A sequence \( (x_n)_{n\in\mathbb{N}} \) in a topological vector space \( X \) is called a \textit{basic sequence} if it is a Schauder basis for its closed linear span, that is, for all $x \in \overline{\operatorname{span}} \{x_n\}$,
\[
x = \sum_{n=1}^\infty a_n x_n
\]
for a unique sequence of scalars \( (a_n)_{n\in\mathbb{N}} \). The sequence is said to be \textit{unconditional} if the convergence of the above series is unconditional, that is, for every permutation \( \pi: \mathbb{N} \to \mathbb{N} \),
\[
\sum_{n=1}^\infty a_n x_n \quad \text{converges if and only if} \quad \sum_{n=1}^\infty a_{\pi(n)} x_{\pi(n)} \quad \text{also converges.}
\]

Moreover, the \textit{Schauder decomposition} in a Banach space can be defined as follows: We say that a Banach space \( X \) admits a Schauder decomposition if there exist closed subspaces \( X_n \subset X \) (\( n \in \mathbb{N} \)) such that:
\begin{itemize}
\item \( X = \bigoplus_{n=1}^\infty X_n \), that is, every element \( x \in X \) can be written uniquely as
\[
x = \sum_{n=1}^\infty x_n, \quad \text{with } x_n \in X_n.
\]
The sum is considered in the normed sense, that is, the series converges in norm.
\item The projection \( P_n: X \to X_n \), defined by \( P_n(x) = x_n \), is linear and continuous for each \( n \in \mathbb{N} \).
\end{itemize}

Obviously, the Schauder decomposition is a generalization of the Schauder basis, since every space with a Schauder basis admits a Schauder decomposition.

Although the SQP remains open for Banach spaces, the analogous question of whether every non-normable Fréchet space (i.e., infinite-dimensional, locally convex, complete, and metrizable space) admits a separable quotient has already been settled. Indeed, if $X$ is a non-normable Fr\'echet space, a result by Eidelheit \cite{Eidelheit} ensures that $X$ has a quotient isomorphic to $\mathbb{K}^{\mathbb{N}}$.

In \cite{Drewnowski}, Drewnowski posed a more general question: Does every complete, metrizable, infinite-dimensional topological vector space $X$ contain a closed subspace $Y$ such that the dimension of the quotient $X/Y$ is the cardinality of the continuum (in short, $\dim(X/Y) = \mathfrak{c}$)? The same paper includes a remark stating that every infinite-dimensional Fr\'echet space admits a quotient of dimension $\mathfrak{c}$. Drewnowski's problem was resolved by Popov (see \cite{Popov}). He showed that for $0 < p < 1$, the space $X := L_p([0,1]^{2^{\mathfrak{c}}})$ does not admit a suitable closed subspace $Y$ such that $\dim(X/Y) \leq \mathfrak{c}$. Consequently, $X$ does not admit a separable quotient.

Above, as usual, for $p\in(0,\infty)$, $L_{p}(\Omega)$ will denote the space of all (Lebesgue classes of) measurable functions $f:\Omega \rightarrow\mathbb{K}$ such that $|f|^{p}$ is $\mu$-integrable on $\Omega$, where $(\Omega,\mathcal{M},\mu)$ is a measure space and $\mu$ is a positive measure. Recall that if $1\leq p<\infty$ (if $0<p<1$) then $L_{p}(\Omega)$ becomes a Banach (quasi-Banach, resp.) space under the norm ($p$-norm, resp.) $\Vert f\Vert_{p}=(\int_{\Omega}|f|^{p}d\mu)^{1/p}$.

This paper is structured as follows. In Section \ref{sec2}, we present key preliminaries and known equivalences related to the SQP. Section \ref{sec3} introduces a novel approach by connecting the SQP with semigroup theory, proving that every infinite-dimensional Banach space that does not possess the Lotz property admits a separable quotient. As a consequence, for instance, every Banach lattice space with a quasi-interior point admits a separable quotient. In Section \ref{sec4}, we explore partial converses of the main result of Section \ref{sec3}, analyzing conditions on Banach spaces under which the existence of separable quotients implies the absence of the Lotz property. Finally, Section \ref{sec5} introduces new sufficient or equivalent conditions for the SQP, particularly through the lens of operator ranges and continuous seminorms.

\section{Preliminaries}\label{sec2}

We will start by presenting some equivalences related to the SQP, which will be useful for us for obtaining new tools and generalizations. The equivalence of (2) and (3) below, even for locally convex Hausdorff spaces, is an obvious consequence of the Bipolar Theorem (see also \cite{Kakol2014,saxon2016,saxon2017}). The equivalence of (1) and (4) is due to Saxon--Wilansky \cite{saxon1977}. Recall that a locally convex space $X$ is called \textit{barreled} if every \textit{barrel} --a closed, convex, balanced and absorbing set in $X$ (see \cite{narici})-- is a neighborhood of the origin. See also \cite{bennet1973,johnson1972} for some partial results related to the next theorem.

\begin{theorem}[Saxon--Wilansky \cite{saxon1977}]\label{t1}
The following statements are equivalent for an infinite-dimensional Banach space \( X \):
\begin{enumerate}
\item \( X \) contains a dense non-barreled subspace.
\item \( X \) admits a strictly increasing sequence of closed subspaces whose union is dense in \( X \).
\item \( X^* \) admits a strictly decreasing sequence of weak$^*$ ($\mathrm{w}^*$ for short)-closed subspaces whose intersection consists only of the zero element.
\item \( X \) has a separable quotient.
\end{enumerate}
\end{theorem}

Let \( X \) be a Banach space, we denote by \( S(X) \) the set of all vectors in \( X \) with norm $1$, and by \( X^* \) we denote the dual of $X$, that is, the vector space consisting of all functionals \( \displaystyle \varphi: X \to \mathbb{K} \) that are linear and continuous, endowed with the norm \( \|\varphi\| = \sup_{x \in S(X)} |\varphi(x)| \). We say that a sequence \( \displaystyle (y_n)_{n\in\mathbb{N}} \subset S(X^*) \) is \textit{normal} if \( \displaystyle \lim_{n \to \infty} y_n(x) = 0 \) for all \( x \in X \); clearly, normal sequences coincide with normalized $\mathrm{w}^*$-null sequences. The Josefson-Nissenzweig Theorem \cite{josefson1975,nissenzweig1975} states that the dual of any infinite-dimensional Banach space contains a normal sequence. It is easy to see that a sequence \( (y_n)_{n\in\mathbb{N}} \subset S(X^*) \) is normal if and only if the subspace \( \displaystyle \{x \in X : \lim_{n \to \infty} y_n(x) = 0\} \) is dense in \( X \).

We will say that a sequence \( (y_n)_{n\in\mathbb{N}} \subset S(X^*) \) is \textit{strongly normal} if the subspace\linebreak \(\displaystyle \{x \in X : \sum_{n=1}^{\infty} |y_n(x)| < \infty\} \) is dense in \( X \) (see \cite{sliwa2012}). Clearly, every strongly normal sequence in \( X^* \) is normal. The following theorem from \cite{sliwa2012} shows a connection between these concepts.

\begin{theorem}[\'Sliwa \cite{sliwa2012}]\label{t2.2}
Let \( X \) be an infinite-dimensional Banach space. The following conditions are equivalent:
\begin{enumerate}
\item \( X \) has a separable quotient.
\item \( X^* \) has a strongly normal sequence.
\item \( X^* \) has a basic sequence in the $\mathrm{w}^*$ topology.
\end{enumerate}
\end{theorem}

From \cite[Theorem 1]{sliwa2012}, we know that every strongly normal sequence in \( X^* \) contains a subsequence \( (y_n)_{n\in\mathbb{N}} \) that is a basic sequence in the $\mathrm{w}^*$ topology. Conversely, every normalized basic Schauder sequence in \( (X^*, \mathrm{w}^*) \) is strongly normal (see \cite[Proposition 2]{sliwa2012}).

It is worth noting that, as shown in \cite{sliwa2012}, if $X$ is an infinite-dimensional weakly compactly generated Banach space, then every normal sequence in $X^*$ contains a strongly normal subsequence. This result provides a partial answer to a question posed by \'Sliwa himself, who asked whether this property holds for arbitrary Banach spaces. Although it demonstrates that a large class of Banach spaces satisfies the property, the general problem remains open.

A slightly stronger property than condition (2) of Theorem \ref{t2.2} is given in \cite{CKLS}. As mentioned in \cite{CKLS}, this property turns out to be equivalent to the existence in \( X^* \) of a basic sequence equivalent to the canonical basis of \( c_0 \).

\begin{proposition}[Ferrando--Kakol--López-Pellicer--\'Sliwa \cite{CKLS}]\label{p2.4}
A Banach space \( X \) has a quotient isomorphic to \( \ell_1 \) if and only if \( X^* \) contains a normal sequence \( (y_n)_{n\in\mathbb{N}} \) such that \( \sum_{n=1}^{\infty} |y_n(x)| < \infty \) for all \( x \in X \).
\end{proposition}

The above proposition will play a crucial role in the proof of the main result of Section \ref{sec4} (Theorem \ref{onemain}).

Another result that will be essential for the proof of Theorem \ref{main} is the following theorem, which is an immediate consequence of \cite[Lemma 3.4]{JIMENEZ} (see also \cite{M}), obtained by simply taking $X = \{ 0\}$ in \cite[Lemma 3.4]{JIMENEZ} as the closed subspace of infinite codimension.

Let $X$ be a Banach space. We say that a subspace $R$ of $X$ is an \textit{operator range} if there exists a Banach space $Z$ and a continuous linear operator $T:Z\longrightarrow X$ such that $T(Z)=R$.

\begin{theorem}\label{t2.9}
A Banach space has a separable quotient if and only if it has a proper and dense operator range.
\end{theorem}

\section{SQP and semigroup theory: A new approach}\label{sec3}

This section is devoted to the study of semigroups of continuous operators, with particular emphasis on the existence of strongly continuous semigroups that are not uniformly continuous.

We begin with a brief review of key concepts and results from semigroup theory that are relevant to our purposes (see \cite{pazy1983}).

From now on, for Banach spaces $X$ and $Y$, we denote by $\mathcal{L}(X,Y)$ the space of all continuous linear operators from $X$ to $Y$. When $X=Y$ we simply write $\mathcal{L}(X)$ instead of $\mathcal{L}(X,Y)$.

\begin{definition}
A semigroup of linear operators on $X$ is a family $(T(t))_{t\geq 0} \subset \mathcal{L}(X)$, denoted simply by $T$ when there is no risk of confusion, that satisfies:
\begin{itemize}
\item $T(0) = I$, where $I$ is the identity operator on $X$;
\item $T(t + s) = T(t)T(s)$, for all $t, s \geq 0$. 
\end{itemize}
If, furthermore, we have
\begin{itemize}
\item $\displaystyle \lim_{t \to 0^+} \|T(t) - I\|_{\mathcal{L}(X)} = 0$, we say that the semigroup is \textit{uniformly continuous};
\item $\displaystyle \lim_{t \to 0^+} \|T(t)x - x\|_X = 0$, for each $x \in X$, we say that the semigroup is \textit{strongly continuous}, or a $C_0$-\textit{semigroup}. If $T \subset \mathcal{L}(X)$ is a $C_0$-semigroup, its infinitesimal generator is the operator $A : D(A) \subset X \to X$ defined by
\[
Ax = \lim_{t \to 0^+} \frac{T(t)x - x}{t},
\]
where
\[
D(A) = \left\{ x \in X : \lim_{t \to 0^+} \frac{T(t)x - x}{t} \text{ exists} \right\}.
\]
\end{itemize}
\end{definition}

Given a $C_0$-semigroup $T \subset \mathcal{L}(X)$, it is well known that its infinitesimal generator $A$ is a closed and densely defined linear operator. It is also well known that $T$ is uniformly continuous if and only if $A$ is bounded.

For further details and basic results on semigroup theory, see, for instance, \cite{pazy1983}.

It is well known that every uniformly continuous semigroup is a $C_0$-semigroup. Moreover, for certain Banach spaces $X$, it is possible to establish the existence of a $C_0$-semigroup on $X$ that is not uniformly continuous. For instance, beyond the wide range of Banach spaces for which it is known that the classes of uniformly continuous semigroups and $C_0$-semigroups do not coincide, we will prove (see Theorem \ref{t14}) that in every Banach space with an unconditional Schauder decomposition these classes do not coincide.

Among the wide range of Banach spaces for which it is known that the classes of uniformly continuous semigroups and $C_0$-semigroups do not coincide, we can mention, for example, Hilbert spaces, $L_p(\Omega)$ spaces, Sobolev spaces, the space of bounded and uniformly continuous functions, separable spaces containing copies of $c_0$, spaces whose topological dual contains a copy of $\ell_1$, WLD spaces with copies of $c_0$, spaces that contain a copy of $c_0$ and whose dual unit ball is sequentially compact in the $\mathrm{w}^*$ topology, among others.

Although we still do not have a general criterion to determine whether a strongly continuous semigroup is uniformly continuous or not, we can verify the existence of non-uniformly continuous semigroups of operators on every Banach space. Before this, recall that if $X$ is a Banach space and $(P_n)_{n\in\mathbb{N}}$ is a sequence in $\mathcal{L}(X)$, then the sequence $(X_n, P_n)_{n\in\mathbb{N}}$, where $X_n:=P_n(X)$ and $$P_n P_m = \delta_{nm} P_{\min\{n,m\}},$$ with $\delta_{nm}$ denoting the Kronecker delta, is called a \textit{generalized biorthogonal system} (see Definition 15.4 of \cite{Singer1981}).

\begin{proposition}
Let $X$ be an infinite-dimensional Banach space. Then there exists a semigroup of continuous operators on $X$ that is not uniformly continuous.
\end{proposition}

\begin{proof}
From the Josefson-Nissenzweig Theorem \cite{josefson1975,nissenzweig1975} and by \cite[Lemma 3.1.19.]{gonzalez2010tauberian}, we obtain a generalized biorthogonal system $(X_n,P_n)_{n\in\mathbb{N}}$ with
\[
\lim_{n \to \infty}\|P_n(x)\|=0
\]
for all $x\in X$. Now, take $(a_n)_{n\in\mathbb{N}}\in \bigcap_{p>0}\ell_{p}$ a sequence of positive terms ($a_n=c^n$ for $0<c<1$ is an example). Let us consider $\alpha_n:[0,\infty)\longrightarrow\mathbb{R}$ given by $\alpha_n(t)=a_n^t$. Note that for each $t>0$ and $x\in X$, we have
\[
\sum_{n=1}^{\infty}\|\alpha_n(t)P_n(x)\|\leq \sup_{n\in \mathbb{N}}\|P_n(x)\| \cdot \sum_{i=1}^{\infty}|a_n|^t<\infty.
\]
Since $X$ is Banach, the family of operators $(T(t))_{t>0}\subset\mathcal{L}(X)$ given by
\[
T(t)(x)=\sum_{n=1}^{\infty}\alpha_n(t)P_n(x)
\]
is well defined. If $t,s>0$, then for $x\in X$ we have
\begin{align*}
T(t)T(s)(x)&=\sum_{i=1}^{\infty}\alpha_i(t)P_i\left(\sum_{j=1}^{\infty}\alpha_j(s)P_j(x)\right)=\sum_{n=1}^{\infty}\alpha_n(t)\alpha_n(s)P_n(x)\\
&=\sum_{n=1}^{\infty}\alpha_n(t+s)P_n(x)=T(t+s)(x).
\end{align*}
Therefore, defining $T(0)=I$, we have that the family $(T(t))_{t\geq 0}$ is a semigroup of continuous operators on $X$. Now observe that for each $n\in\mathbb{N}$ and each $x\in X$, we have
\[
\lim_{t \to 0^+}\frac{T(t)\left(P_n(x)\right)-P_n(x)}{t}=\lim_{t \to 0^+}\frac{\alpha_n(t)-1}{t}P_n(x)=\ln(a_n)P_n(x).
\]

If the semigroup $(T(t))_{t\geq 0}$ were uniformly continuous, we would have $A\left(P_n(x)\right))=\ln(a_n)P_n(x)$ for all $n\in\mathbb{N}$ and for all $x\in X$, whence it would follow $|\ln(a_n)|\leq \|A\|$ for all $n\in\mathbb{N}$. But since $\lim_{n \to \infty}a_n=0$, it follows that $\|A\|=\infty$, which is a contradiction.
\end{proof}

Given an infinite-dimensional Banach space, a natural question arises: Does there exist a semigroup of continuous linear operators that is strongly continuous but not uniformly continuous? This question was addressed by Lotz \cite{lotz}, who investigated conditions under which pointwise convergence of operators implies uniform convergence. As a consequence, it was shown that if the space is a \textit{Grothendieck space} with the \textit{Dunford--Pettis property}, then the classes of strongly continuous and uniformly continuous semigroups coincide (see \cite{lotz}). Recall that a Banach space $X$ is said to be a Grothendieck space if every sequence in $X^*$ that converges in the $\mathrm{w}^*$ topology also converges in the weak topology on $X^*$. Moreover, we say that $X$ has the Dunford--Pettis property if
\[
\lim_{n\to\infty} x_n^*(x_n) = x^*(x),
\]
whenever $(x_n)_{n\in\mathbb{N}} \subset X$ converges weakly to $x$ and $(x_n^*)_{n\in\mathbb{N}} \subset X^*$ converges weakly to $x^*$.

We found \textcolor{black}{only a} few references in the literature addressing the existence of Banach spaces in which these classes of semigroups coincide (see \cite{lotz,neerven,Leung}), and no one discussing whether there is any relationship between such coincidence (or lack thereof) and the existence of separable quotients. This motivates the following definition (see \cite{neerven}):

\begin{definition}
We say that a Banach space $X$ has the Lotz property if every $C_0$-semigroup on $X$ is uniformly continuous
\end{definition}

The main result of \cite{lotz} shows that Grothendieck spaces with the Dunford--Pettis property possess the so-called Lotz property. This result is not optimal in the sense that there exist spaces which do not simultaneously have these two properties, and yet the classes of strongly continuous and uniformly continuous semigroups coincide on such spaces.

This was addressed by Leung \cite{Leung}, where the author was able to weaken the assumption that the space has the Dunford--Pettis property. Although Leung \cite{Leung} provided an explicit example showing that the coincidence of strongly and uniformly continuous semigroups does not imply that the space is Grothendieck and has the Dunford--Pettis property, Van Neerven \cite{neerven} proved that this implication holds when the space is a Banach lattice with a quasi-interior point.

Let us now state and prove the main result of this section.

\begin{theorem}\label{main}
Every infinite-dimensional Banach space that does not possess the Lotz property admits a separable quotient.
\end{theorem}

\begin{proof}
Let \( X \) be a Banach space that does not possess the Lotz property, and let \( T \subset \mathcal{L}(X) \) be a \( C_0 \)-semigroup that is not uniformly continuous. Denote by \( A \) the infinitesimal generator of \( T \).

It is well known that \( A \) is a densely defined linear operator; that is, \( D(A) \) is dense in \( X \). Moreover, since \( A \) is closed, \( D(A) \), endowed with the graph norm
\[
\|x\|_{G(A)} := \|x\| + \|Ax\|,
\]
is itself a Banach space. Since \(\|x\|_{G(A)} \ge \|x\|\) for all \(x \in D(A)\), the inclusion map
\[
i : (D(A), \|\cdot\|_{G(A)}) \longrightarrow (X, \|\cdot\|)
\]
is continuous. Therefore, \( D(A) \) can be regarded as the image of a continuous linear operator between Banach spaces, and thus as an operator range.
	
Finally, recall that a \( C_0 \)-semigroup is uniformly continuous if and only if its infinitesimal generator is bounded. Since \( T \) is not uniformly continuous, its generator \( A \) must be unbounded, implying that \( D(A) \) is a proper subspace of \( X \). Indeed, assume (by  contradiction) that \( D(A) = X \). In that case, the inclusion map \( i \) would be a continuous and surjective linear operator between Banach spaces. By the \textit{Closed Graph Theorem}, its inverse would also be continuous, meaning that the graph norm \(\|\cdot\|_{G(A)}\) would be equivalent to the original norm on \(X\). Consequently, \( A \) would be bounded, contradicting the assumption that \(T\) is not uniformly continuous. Thus, \( D(A) \) is indeed a proper subspace of \( X \).

Therefore, \( D(A) \subset X \) is a proper and dense operator range. From Theorem \ref{t2.9}, we conclude that \( X \) admits a separable quotient, which completes the proof.
\end{proof}

\begin{remark}
The converse of the previous theorem does not hold. In fact, the result of Lotz \cite{lotz} provides the necessary counterexample since it can be shown that any Banach space that is either a Grothendieck space or has the Dunford--Pettis property admits a separable quotient. In fact, if the space is Grothendieck and $B_{X^*}$ is not sequentially compact in the $\mathrm{w}^*$ topology, we can  conclude that the space has a separable quotient by \cite{CKLS}. If $B_{X^*}$ is sequentially compact in the $\mathrm{w}^*$ topology, we have that $B_{X^*}$ is a reflexive space by the Erberlein--Smulian Theorem (see \cite{abk}), and then we guarantee the  separable quotient, for instance, by \cite{pelczynski1964}. Now, if the space has the Dunford--Pettis property, it is possible to prove that $X$ or $X^*$ contains a copy of $\ell_1$. Indeed, otherwise, by the Rosental's $\ell_1$-Theorem (see \cite{abk}), we would be able to construct sequences $(x_n)_{n\in\mathbb{N}}$ and $(x_n^*)_{n\in\mathbb{N}}$ in $X$ and $X^*$ respectively, with $x_n \overset{\mathrm{w}}{\to} 0 $, $x_n^* \overset{\mathrm{w}}{\to} 0 $ and $x_n^*(x_n) =1$, for all $n\in \mathbb{N}$.
\end{remark}

In view of the theorem above, to address the SQP now it suffices to restrict our attention to Banach spaces with the Lotz property.

Since, in the context of Banach lattice spaces with quasi-interior point (see \cite{meyer}), possessing the Lotz property is equivalent to being a Grothendieck space with the Dunford--Pettis property (see \cite{neerven} and Figure \ref{fig}), we obtain the following result:

\begin{corollary}
Every Banach lattice space with a quasi-interior point has a separable quotient.
\end{corollary}

In view of everything discussed above, we conjecture that the SQP holds in general.

By narrowing the scope of the problem to spaces with the Lotz property, we aim to either generalize existing results or establish new positive results for certain classes of spaces that have not yet been investigated.

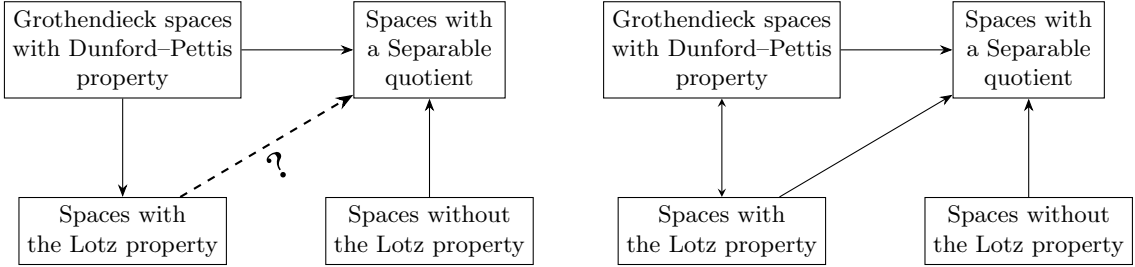
\begin{figure}[h!]
\begin{footnotesize}
\begin{tikzpicture}[
box/.style={draw, minimum width=2cm, minimum height=.8cm, align=center},
arrow/.style={-Stealth, thick},
dashedarrow/.style={-Stealth, thick, dashed},
node distance=1.3cm and 1.3cm
]

\node[box] (grodp) {Grothendieck spaces \\ with Dunford--Pettis \\ property};
\node[box, below=of grodp] (lotz) {Spaces with \\ the Lotz property};
\node[box, right=of lotz] (naolotz) {Spaces without \\ the Lotz property};
\node[box, above=of naolotz] (sq) {Spaces with \\ a Separable \\ quotient};

\draw[-Stealth] (grodp) -- (lotz);
\draw[-Stealth] (grodp) -- (sq);
\draw[-Stealth] (naolotz) -- (sq);
\draw[dashedarrow] (lotz) -- node[midway, below, sloped] {{\textbf {\Large ?}}} (sq);
\end{tikzpicture}
\hspace{.7cm}
\begin{tikzpicture}[
box/.style={draw, minimum width=2cm, minimum height=.8cm, align=center},
arrow/.style={-Stealth, thick},
dashedarrow/.style={-Stealth, thick, dashed},
node distance=1.3cm and 1.3cm
]

\node[box] (grodp) {Grothendieck spaces \\ with Dunford--Pettis \\ property};
\node[box, below=of grodp] (lotz) {Spaces with \\ the Lotz property};
\node[box, right=of lotz] (naolotz) {Spaces without \\ the Lotz property};
\node[box, above=of naolotz] (sq) {Spaces with \\ a Separable \\ quotient};

\draw[stealth-stealth] (grodp) -- (lotz);
\draw[-Stealth] (grodp) -- (sq);
\draw[-Stealth] (naolotz) -- (sq);
\draw[-Stealth] (lotz) -- (sq);
\end{tikzpicture}
\end{footnotesize}
\caption{On the left, the diagram represents the problem for an arbitrary Banach space, while on the right, it illustrates the scheme for Banach lattice spaces with a quasi-interior point}
\label{fig}
\end{figure}

\section{Partial converses}\label{sec4}

Although the converse of Theorem \ref{main} does not hold in general, one can consider this problem within more restricted classes of Banach spaces and, thus, obtain partial converses of Theorem \ref{main}. For notational convenience, from now on we shall write $T_t$ instead of $T(t)$. Our goal now is to develop a method to construct semigroups in arbitrary Banach spaces. To do so, we will start by assuming the existence of a semigroup of the form
\begin{equation}\label{eq1}
T_t(x) = \alpha(t)x + \sum_{i=1}^{\infty}\alpha_i(t)P_i(x),
\end{equation}
where, for all $n \in \mathbb{N}$, the functions $\alpha,\alpha_n: [0, \infty) \longrightarrow \mathbb{K}$ and the operato\textcolor{black}{r}s $P_n\in\mathcal{L}(X)\setminus\{0\}$ fulfill certain conditions, which will be stated and verified in due course.

It is important to note that the convergence of the series in equation \eqref{eq1} does not need to occur in the strong topology of $X$ (norm topology), but rather in the weak topology. Indeed, it is not dif\textcolor{black}{f}icult to prove the following result:

\begin{proposition}\label{p3.7}
Let $X$ and $Y$ Banach spaces and $(T_n)_{n\in\mathbb{N}}$ a sequence in $\mathcal{L}(X,Y)$. If for each $x\in X$ the sequence $(T_n(x))_{n\in\mathbb{N}}$ converges weakly in $Y$, then the limit operator
\[
T(x)=\lim_{n\to\infty}T_n(x),
\]
is continuous.
\end{proposition}


We will suppose that the sequence $(P_n)_{n\in\mathbb{N}}$ in \eqref{eq1} verify
\begin{equation}\label{eq2}
P_n P_m = \delta_{nm} P_{\min\{n,m\}}.
\end{equation}
In this case, using the expression of \eqref{eq1} we have, for $t,s\geq 0$,
\begin{align*}
(T_t T_s)(x)&=\alpha(t)\alpha(s) x + \sum_{n=1}^{\infty}\alpha(t)\alpha_n(s)P_n(x)+\sum_{n=1}^{\infty}\alpha(s)\alpha_n(t)P_n(x)\\
&\qquad\qquad+\sum_{n=1}^{\infty}\alpha_n(t)P_n\left(\sum_{j=1}^{\infty}\alpha_j(s)P_j(x)\right)\\
&=\alpha(t)\alpha(s) x+\sum_{n=1}^{\infty}\left(\alpha(t)\alpha_n(s)+\alpha(s)\alpha_n(t)+\alpha_n(t)\alpha_n(s)\right)P_n(x).
\end{align*}
Since we want the equality $(T_t T_s)(x)=T_{t+s}(x)$ for all $x\in X$, we will require that the coefficients in the expressions of $(T_t T_s)(x)$ and $T_{t+s}(x)$ be equal, that is,
\[
\begin{cases}
\alpha(t+s)=\alpha(t)\alpha(s), \ \text{with} \ \alpha(0)=1, \\
\alpha_n(t+s)=\alpha(t)\alpha_n(s)+\alpha(s)\alpha_n(t)+\alpha_n(t)\alpha_n(s),\ \text{with}\ \alpha_n(0)=0, \ \forall n\in\mathbb{N}.
\end{cases}
\]

Solutions of $\alpha(t+s)=\alpha(t)\alpha(s)$ are of the form $\alpha(t)=a^{t}$ with $a>0$. Now, for each chosen solution of $\alpha$, the solutions $\alpha_n$ are given by 
\begin{equation}\label{alphan}
\alpha_n(t)=a_n^{ t}-\alpha(t) \ \text{ for } \ a_n>0.
\end{equation}
Thus, for each choice of $\alpha$, we have an infinity of choices for $\alpha_n$ for each $n\in\mathbb{N}$. Of course, we must choose these sequences aiming at the convergence of \eqref{eq1}.

Since for each $n\in\mathbb{N}$, the projection $P_n$ is non-zero, we can take $y_n\in \mathrm{Im}(P_n)\setminus\{0\}$ and note that
\[
T_t(y_n)=\alpha(t)y_n+\alpha_n(t)P_n(y_n)=(\alpha_n(t)+\alpha(t))y_n.
\]
Thus
\[
A(y_n)=\lim_{t \to 0^+}\frac{T_t(y_n)-y_n}{t}=\lim_{t \to 0^+}\frac{\alpha_n(t)+\alpha(t)-1}{t}y_n=(\alpha'_n(0)+\alpha'(0))y_n,
\]
and so we have
\[
|\ln(a_n)|=|\alpha'_n(0)+\alpha'(0)|=\frac{\|A(y_n)\|}{\|y_n\|}.
\]
If the infinitesimal generator of $(T_t)_{t\geq 0}$ is continuous, that is, $(T_t)_{t\geq 0}$ is uniformly continuous, we have
\begin{equation}\label{eq4}
\sup_{n\in \mathbb{N}}|\ln(a_n)|\leq \|A\|.
\end{equation}

Now we can state our first partial converse of Theorem \ref{main}.

\begin{theorem}\label{onemain}
Let $X$ be an infinite-dimensional Banach space. If $X$ has quotient isomorphic to $\ell_1$, then $X$ fails to have the Lotz property; that is, there exists a semigroup \( T \) of continuous linear operators which is strongly continuous but not uniformly continuous.
\end{theorem}

\begin{proof}
By Proposition \ref{p2.4}, there exists $(\varphi_n)_{n\in\mathbb{N}}\subset S(X^*)$ such that $\sum_{n=1}^{\infty}|\varphi_n(x)|<\infty$ for all $x\in X$. By \cite[Lemma 3.1.19.]{gonzalez2010tauberian}, taking a subsequence of $(\varphi_n)_{n\in\mathbb{N}}$ if necessary, we can assume that there exists a bounded sequence $(x_n)_{n\in\mathbb{N}}$ such that $\varphi_n(x_m)=\delta_{nm}$. Consider now the sequence $(P_n)_{n\in\mathbb{N}}$ in $\mathcal{L}(X)$ given by $P_n(x)=\varphi_n(x)x_n$. We obviously have $P_n P_m=\delta_{nm} P_{\min\{n,m\}}$ and
\[
\sum_{n=1}^{\infty}\|P_n(x)\|=\sum_{n=1}^{\infty}|\varphi_n(x)|\|x_n\|\leq \left(\sup_{n\in \mathbb{N}}\|x_n\|\right)\cdot\left(\sum_{n=1}^{\infty}|\varphi_n(x)|\right)<\infty
\]
for all $x\in X$. Let $(a_n)_{n\in\mathbb{N}}$ a sequence of positive numbers such that $\lim_{n \to \infty}a_n=0$. Thus we have $\lim_{n \to \infty}|\ln (a_n)|=\infty$ and $\lim_{n \to \infty}\alpha_n(t)=-\alpha(t)$ for all $t>0$, where $\alpha_n$ is of the form given in \eqref{alphan} and $\alpha(t)=a^t$ for some $a>0$. Furthermore, there exists $K>0$ such that $|\alpha_n(t)|\leq K$ for all $t\in [0,1]$ and all $n\in\mathbb{N}$. Since for each $x\in X$ and $t>0$ the series
\[
\sum_{i=1}^{\infty}\alpha_i(t)P_i(x)
\]
is absolutely convergent, it follows that \eqref{eq1} is well defined. Now, given $\epsilon>0$ and $x\in X$, take $N\in\mathbb{N}$ such that
\[
\sum_{n=N}^{\infty}\|P_n(x)\|\leq\frac{\epsilon}{2K}.
\]
Now take $\delta>0$ such that
\[
|\alpha(t)-1|\leq \frac{\epsilon}{2N(\|x\|+1)}
\]
and
\[
|\alpha_n(t)|\leq \frac{\epsilon}{2N(\|P_n(x)\|+1)}
\]
for all $n< N$ and $t<\delta$. Therefore, for $t<\delta$ we have
\begin{align*}
\|T_t(x)-x\|&=\left\|(\alpha(t)-1)x+\sum_{n=1}^{\infty}\alpha_n(t)P_n(x)\right\|\leq |\alpha(t)-1|\cdot \|x\|+\sum_{n=1}^{\infty}|\alpha_n(t)|\cdot\|P_n(x)\|\\
&\leq|\alpha(t)-1|\cdot \|x\|+\sum_{n=1}^{N-1}|\alpha_n(t)|\cdot\|P_n(x)\|+\sum_{n=N}^{\infty}K\cdot\|P_n(x)\|\\
&\leq \frac{\epsilon\cdot\|x\|}{2N(\|x\|+1)}+\sum_{n=1}^{N-1}\frac{\epsilon\cdot\|P_n(x)\|}{2N(\|P_n(x)\|+1)}+\frac{\epsilon}{2}\leq N\cdot\frac{\epsilon}{2N}+\frac{\epsilon}{2}=\epsilon,
\end{align*}
that is, $\displaystyle \lim_{t \to 0^+}T_t(x)=x$ for each $x\in X$. We then have that $(T_t)_{t\geq 0}$ given by \eqref{eq1} is a $C_0$-semigroup. However, the fact that $\lim_{n \to \infty}|\ln(a_n)|=\infty$ guarantees us that $(T_t)_{t\geq 0}$ is not uniformly continuous, due to the inequality seen in \eqref{eq4}.
\end{proof}

In fact, in the above theorem we need something weaker than having a quotient isomorphic to $\ell_1$, which is the existence of an unconditional Schauder decomposition.

It is well known that every weakly sequentially complete Banach space admits a separable quotient (as it is either reflexive or contains a copy of $\ell_1$). If such a space additionally satisfies a certain property, one can establish the existence of a strongly continuous semigroup that is not uniformly continuous. This result is stated in the following theorem (Theorem \ref{t3.10}). For its proof, it is useful to recall the following characterization:

\begin{remark}\label{ob3.9}
It is known that a semigroup $(T_t)_{t\geq0}$ is strongly continuous if and only if it is weakly continuous (see \cite{pazy1983}), that is,
\[
\lim_{t \to 0^+}\varphi(T_t(x)-x)=0
\]
for all $(x,\varphi)\in X\times X^*$.
\end{remark}

\begin{theorem}\label{t3.10}
Let $X$ be a weakly sequentially complete Banach space. If there exists a sequence of projections $(P_n)_{n\in\mathbb{N}}\subset\mathcal{L}(X)\setminus\{0\}$ such that $P_n P_m=\delta_{nm}P_{\min\{n,m\}}$ and
\[
\sum_{n=1}^{\infty}|\varphi(P_n(x))|<\infty
\]
for all $(x,\varphi)\in X\times X^*$, then $X$ fails to have the Lotz property.
\end{theorem}

\begin{proof}
Consider $\alpha(t)=a^{t}$ for any $a\in (0,\infty)$ and for each $n\in\mathbb{N}$ take $\alpha_n(t)=\left(\frac{1}{n}\right)^{t}-\alpha(t)$. For each $n\in\mathbb{N}$ consider $T_t^{(n)}\in \mathcal{L}(X)$ given by $T_t^{(n)}(x)=\alpha(t)x+\sum_{i=1}^{n}\alpha_i(t)P_i(x)$. Let us see that the sequence $(T_t^{(n)}(x))_{n\in\mathbb{N}}$ converges weakly for all $x\in X$ and all $t\geq 0$. Since the sequence $(\alpha_n(t))_{n\in\mathbb{N}}$ is bounded, we have for $\varphi\in X^*$ that
\[
\sum_{i=1}^{\infty}|\alpha_i(t)\varphi\left(P_i(x)\right)|<\infty.
\]
Therefore $\sum_{i=1}^{\infty}\alpha_i(t)\varphi\left(P_i(x)\right)$ converges in $\mathbb{K}$. Thus
\[
\lim_{n \to \infty}\varphi\left(T_t^{(n)}(x)\right)=\lim_{n \to \infty}\left(\alpha(t)\varphi(x)+\sum_{i=1}^{n}\alpha_i(t)\varphi(P_i(x))\right)
\]
exists for any $\varphi\in X^*$. Since $X$ is weakly sequentially complete, we can conclude that there exists $T_t(x)\in X$ such that the sequence $(T_t^{(n)}(x))_{n\in\mathbb{N}}$ converges weakly to $T_t(x)$. Proposition \ref{p3.7} assures us that $T_t:X\longrightarrow X$ is continuous.

Let us now see that
\[
T_t^{(n)} T_s^{(m)}=T_{t+s}^{(\min\{n,m\})}+\alpha(t)\left(T_{s}^{(m)}-T_{s}^{(\min\{n,m\})}\right)+\alpha(s)\left(T_{t}^{(n)}-T_{t}^{(\min\{n,m\})}\right)
\]
for any $t,s\geq 0$ and $n,m\in\mathbb{N}$. In fact, given $x\in X$, it follows that
\begin{align*}
&T_t^{(n)}\left(T_s^{(m)}(x)\right)=\alpha(t)T_s^{(m)}(x)+\sum_{i=1}^{n}\alpha_i(t)P_i\left(T_s^{(m)}(x)\right)\\
&=\alpha(t)\left(\alpha(s)x+\sum_{j=1}^{m}\alpha_j(s)P_j(x)\right)+\sum_{i=1}^{n}\alpha_i(t)P_i\left(\alpha(s)x+\sum_{j=1}^{m}\alpha_j(s)P_j(x)\right)\\
&=\alpha(t)\left(\alpha(s)x+\sum_{j=1}^{m}\alpha_j(s)P_j(x)\right)+\alpha(s)\left(\sum_ {i=1}^{n}\alpha_i(t)P_i(x)\right)+\sum_{i=1}^{\min\{n,m\}}\alpha_i(t)\alpha_i(s)P_i(x)\\
&=\alpha(t+s)x+\sum_{i=1}^{\min\{n,m\}}\left(\alpha(t)\alpha_i(s)+\alpha(s)\alpha_i(t)+\alpha_i(t)\alpha_i(s)\right)P_i(x)\\
&\qquad+\alpha(t)\left(\sum_{i=\min\{n,m\}+1}^{m}\alpha_i(s)P_i(x)\right)+\alpha(s)\left(\sum_{i=\min\{n,m\}+1}^{n}\alpha_i(t)P_i(x)\right)\\
&=T_{t+s}^{(\min\{n,m\})}(x)+\alpha(t)\left(T_{s}^{(m)}(x)-T_{s}^{(\min\{n,m \})}(x)\right)+\alpha(s)\left(T_{t}^{(n)}(x)-T_{t}^{(\min\{n,m\})}(x)\right).
\end{align*}
Thus, for each $x\in X$ and $m\in\mathbb{N}$ we have
\begin{align*}
T_t\left(T_s^{(m)}(x)\right) &=\lim\limits_{n\rightarrow\infty}T_t^{(n)}\left(T_s^{(m)}(x)\right)\\
&=\lim\limits_{n\rightarrow\infty}\left[T_{t+s}^{(\min\{n,m\})}(x)+\alpha(t)\left(T_{s}^{(m)}(x)-T_ {s}^{(\min\{n,m\})}(x)\right)\right. \\
& \qquad\qquad \left. +\alpha(s)\left(T_{t}^{(n)}(x)-T_{t}^{(\min\{n,m\})}(x)\right)\right]\\
&=T_{t+s}^{(m)}(x)+\alpha(s)\left(T_t(x)-T_t^{(m)}(x)\right).
\end{align*}
If we make $m\rightarrow\infty$ in the last equality, we obtain $T_t\left(T_s(x)\right)=T_{t+s}(x)$. We then conclude that $(T_t)_{t\geq0}$ is a semigroup of continuous operators of $X$.

Now we will show that $(T_t)_{t\geq0}$ is strongly continuous using Remark \ref{ob3.9}. In fact, for $\varphi\in X^*$ and $x\in X$, given $\epsilon>0$ take $N\in \mathbb{N}$ such that
\[
\sum_{i=N+1}^{\infty}|\varphi(P_i(x))|<\frac{\epsilon}{2K} \ \text{ with } \ K=\sup\limits_{t\in[0,1]}|\alpha(t)|+1.
\]
Now take $0<\delta\leq 1$ such that for $0\leq t<\delta$ we have
\[
|\alpha(t)-1||\varphi(x)|+\sum_{i=1}^{N}|\alpha_i(t)||\varphi(P_i(x))|<\frac{\epsilon}{2}.
\]
Thus, for $0\leq t<\delta$ we have
\begin{align*}
|\varphi\left(T_t(x)-x\right)|&=\left|(\alpha(t)-1)\varphi(x)+\sum_{i=1}^{\infty}\alpha_i(t)\varphi(P_i(x))\right|\\
&\leq|\alpha(t)-1||\varphi(x)|+\sum_{i=1}^{N}|\alpha_i(t)||\varphi(P_i(x))|+\sum_{i=N+1}^{\infty}|\alpha_i(t)||\varphi(P_i(x))|\\
&\leq|\alpha(t)-1||\varphi(x)|+\sum_{i=1}^{N}|\alpha_i(t)||\varphi(P_i(x))|+\sum_{i=N+1}^{\infty}K|\varphi(P_i(x))|\\
&<\frac{\epsilon}{2}+\frac{\epsilon}{2}=\epsilon,
\end{align*}
that is, we verify that $\lim\limits_{t\rightarrow 0}\varphi(T_t(x)-x)=0$ and hence $(T_t)_{t\geq0}$ is indeed a strongly continuous semigroup. It is easy to check that $\mathrm{Im}(P_n)\subset {D}\left(A\right)$ (where $A:{D}(A)\longrightarrow X$ is the infinitesimal generator of the semigroup $(T_t)_{t\geq0}$) for all $n\in\mathbb{N}$ and that $A(P_n(x))=\ln\left(\frac{1}{n}\right)P_n(x)$ for all $x\in X$. This shows us that $(T_t)_{t\geq0}$ is not uniformly continuous.
\end{proof}

\begin{example}
Consider $\ell_p$ for $1<p<\infty$ and take the biorthogonal system $(e_n,e_n^*)_{n\in\mathbb{N}}\subset \ell_p\times (\ell_p)^*=\ell_p\times \ell_{p^*}$, where $(e_n)_{n\in\mathbb{N}}$ is the canonical basis of $\ell_{p}$ and $p^*$ is the H\"older conjugate of $p$, that is, $1/p+1/p^*=1$. Given $(x,y)\in \ell_{p}\times \ell_{p^*}$, we have
\[
\sum_{i=1}^{\infty}|y(e_i)e_i^*(x)|=\sum_{i=1}^{\infty}|y_ix_i|\leq \|x\|_p\|y\|_{p^*}<\infty.
\]
Since $\ell_p$ is reflexive, and every reflexive space is weakly sequentially complete, it follows that $\ell_p$ admits a strongly continuous semigroup that is not uniformly continuous.
\end{example}

\begin{example}
Every Banach space with unconditional basis admits a sequence of orthogonal projections satisfying the condition of Theorem \ref{t3.10}. Indeed, let $(e_n,e_n^*)_{n\in\mathbb{N}}\subset X\times X^*$ be an unconditional basis for $X$. Given $(x,\varphi)\in X\times X^*$, we have for any bijection $\sigma:\mathbb{N}\longrightarrow\mathbb{N}$ that
\[
\sum_{n=1}^{\infty}\varphi\left(e_{\sigma(n)}^*(x)e_{\sigma(n)}\right)=\varphi\left(\sum_{n=1}^{\infty}e_{\sigma(n)}^*(x)e_{\sigma(n)}\right)=\varphi(x),
\]
that is, the series $\sum_{n=1}^{\infty}\varphi\left(e_{n}^*(x)e_{n}\right)$ is unconditionally convergent. Since unconditionally convergent series in $\mathbb{K}$ are exactly those that are absolutely convergent, we conclude that $\sum_{n=1}^{\infty}\left|\varphi\left(e_{n}^*(x)e_{n}\right)\right|<\infty$ for any $(x,\varphi)\in X\times X^*$. It suffices then to take $P_n(x)=e_n^*(x)e_n$.
\end{example}

\begin{example}
It is known that $\ell_1$ is a weakly sequentially complete space (since it is a Schur space). Since the basis of $\ell_1$ is unconditional, it follows from the previous example and Theorem \ref{t3.10} that $\ell_1$ admits a $C_0$-semigroup that is not uniformly continuous.
\end{example}

We will now look at the notion of Schauder decomposition, and how such decompositions allow us to construct strongly continuous semigroups that are not uniformly continuous. The Schauder decomposition in a Banach space is a powerful tool in Functional Analysis, used to represent elements of the space as combinations of indexed sequences of elements of the space.

The interesting thing about Schauder decompositions is that we can have a greater variety of spaces with such decompositions, that is, we are not limited to only separable spaces as in the case of Schauder basis. However, there are still spaces that do not admit Schauder decompositions (not even imposing separability on the space), as is the case of $\ell_\infty$. Just as in the case of unconditional basis, we also have the concept of unconditional Schauder decomposition: The decomposition $(P_n(X))_{n\in\mathbb{N}}$ is said to be an \textit{unconditional Schauder decomposition} if the series $\sum_{n=1}^\infty P_n(x)$ converges unconditionally for any $x\in X$.

Just like Schauder basis, Schauder decompositions have their characterizations in an analogous way. Here we will state only one characterization that will be useful to us.

\begin{theorem}
Let $X$ be a Banach space. The following statements are equivalent:
\begin{itemize}
\item[(1)] $(P_n(X))_{n\in\mathbb{N}}$ is an unconditional Schauder decomposition for $X$;
\item[(2)] There exists a constant $1\leq M<\infty$ such that
\[
\left\|\sum_{i=1}^{\infty}\alpha_iP_i(x)\right\|\leq M\cdot \sup_{n\in \mathbb{N}}|\alpha_n|\cdot\left\|\sum_{i=1}^{\infty}P_i(x)\right\| \]
for all $x\in X$ and $(\alpha_n)_{n\in\mathbb{N}}\in c_{00}$.
\end{itemize}
\end{theorem}

Using arguments similar to those in Theorem \ref{t3.10}, we can establish Theorem \ref{t14} below. Indeed, under the assumptions of Theorem \ref{t14}, and considering the functions $\alpha$ and $\alpha_n$ as defined in Theorem \ref{t3.10}, the convergence of \eqref{eq1} in the norm topology for all $x\in X$ follows from the inequality below, which is ensured by the unconditional Schauder decomposition:
\[
\left\|\sum_{i=1}^{\infty}\alpha_i(t) P_i(x)\right\|\leq M\cdot \sup_{n\in \mathbb{N}}|\alpha_n(t)|\cdot\left\|\sum_{i=1}^{\infty}P_i(x)\right\|.
\]

\begin{theorem}\label{t14}
Let $X$ be a Banach space with a\textcolor{black}{n} unconditional Schauder decomposition. Then $X$ fails to have the Lotz property.
\end{theorem}

In light of the proof of Theorem \ref{t3.10}, a natural question is whether Theorem \ref{t14} can be derived under weaker assumptions. However, the condition in Theorem \ref{t3.10} is precisely equivalent to the unconditionality of the decomposition, as shown below.

First, recall that the concept of a Schauder decomposition can be defined in arbitrary topological vector spaces, and that, in particular, as shown in \cite{pazy1983}, the Schauder decompositions of a Banach space $X$ with respect to the strong and weak topologies coincide.

\begin{proposition}
A Schauder decomposition $(P_n(X))_{n\in\mathbb{N}}$ is unconditional if, and only if,
\[
\sum_{n=1}^{\infty}|\varphi(P_n(x))|<\infty
\]
for all $(x,\varphi)\in X\times X^*$.
\end{proposition}

\begin{proof}
If the series $\sum_{i=1}^\infty P_i(x)$ is unconditionally convergent, then the series $\sum_{i=1}^\infty \varphi(P_i(x))$ is also unconditionally convergent. Since the unconditionally convergent numerical series are exactly those that are absolutely convergent, it follows that $\sum_{i=1}^\infty |\varphi(P_i(x))|<\infty$.

Since $\sum_{i=1}^\infty |\varphi(P_i(x))|<\infty$, for any bijection $\sigma:\mathbb{N}\longrightarrow\mathbb{N}$ we have $\sum_{i=1}^\infty |\varphi(P_{\sigma(i)}(x))|<\infty$. Thus, $\sum_{i=1}^\infty P_{\sigma(i)}(x)$ converges to $x$ in the weak topology, and this tells us that $(P_{\sigma(i)}(X))_{x\in\mathbb{N}}$ is a weak Schauder decomposition, but as noted earlier, this is a Schauder decomposition.
\end{proof}

Finally, we present a result whose verification, although not straightforward, can be established by suitably adapting the techniques and arguments introduced in the preceding results.

\begin{proposition}
Let $X$ be a Banach space. If $X$
\begin{itemize}
\item[(a.)] admits a quotient isomorphic to $\ell_1$; or
\item[(b.)] is weakly sequentially complete and there exists a generalized biorthogonal system $(X_n,P_n)_{n\in\mathbb{N}}$ satisfying $\sum_{n=1}^{\infty}|\varphi\left(P_n(x)\right)|<\infty$ for all $(x,\varphi)\in X\times X^*$,
\end{itemize}
then $X$ admits an unconditional Schauder decomposition. In particular, $X$ fails to have the Lotz property.
\end{proposition}

\section{Another sufficient or equivalent conditions}\label{sec5}

In this section, we establish several new sufficient or equivalent conditions for the SQP.

\begin{lemma}\label{t2.8}
Let $X$ be a Banach space. If $(p_\lambda)_{\lambda\in \Lambda}$ is a family of continuous seminorms, then the subspace $Y=\{x\in X: \sup_{\lambda\in \Lambda}p_\lambda(x)<\infty\}$ is an operator range.
\end{lemma}

\begin{proof}
Evidently $Y$ is a subspace of $X$. Define in $Y$ the following norm
\[
\|x\|_Y=\|x\|+\sup_{\lambda\in \Lambda}p_\lambda(x).
\]
Let us see that $(Y,\|\cdot\|_Y)$ is complete. Let $(y_n)_{n\in\mathbb{N}}$ be a Cauchy sequence in $(Y,\|\cdot\|_Y)$. We then have that $(y_n)_{n\in\mathbb{N}}$ is a Cauchy sequence in $X$, since $\|x\|_y\geq \|x\|$ is valid for all $x\in Y$. Let $y\in X$ be the limit of $(y_n)_{n\in\mathbb{N}}$ in $X$. As $(\|y_n\|_Y)_{n\in\mathbb{N}}$ is bounded (since it is convergent), there is $K>0$ such that
\[
\|y_n\|+\sup_{\lambda\in \Lambda}p_\lambda(y_n)\leq K
\]
for all $n\in\mathbb{N}$. In particular, we have $p_\lambda(y_n)\leq K$ for all $n\in\mathbb{N}$, $\lambda\in \Lambda$. As each $p_\lambda$ is continuous and $y_n\longrightarrow y$ in $X$, we have that $p_\lambda(y)\leq K$ for all $\lambda\in \Lambda$, that is, $y\in Y$. Now, given $\epsilon>0$, there exists $N\in\mathbb{N}$ such that
\[
\|y_n-y_m\|_Y<\epsilon,\ \forall n,m\geq N.
\]
In particular, for each $\lambda\in \Lambda$ we have
\[
\|y_n-y_m\|+p_\lambda(y_n-y_m)<\epsilon,\ \forall n,m\geq N.
\]
Therefore, if $m\longrightarrow\infty$ in the last inequality, we have
\[
\|y_n-y\|+p_\lambda(y_n-y)\leq\epsilon
\]
for all $\lambda\in \Lambda$ and $n\geq N$, that is, $\|y_n-y\|_Y\leq \epsilon$ for $n\geq N$. Therefore $Y$ is complete with the norm $\|\cdot\|_Y$ and is the image of the inclusion $I:(Y,\|\cdot\|_Y)\longrightarrow (X,\|\cdot\|)$ which is continuous.
\end{proof}

From Theorem \ref{t2.9} we know that a Banach space has a separable quotient if and only if it has a proper and dense operator range. With this, we obtain a new equivalence for the SQP in terms of continuous seminorms.

\begin{theorem}
A Banach space $X$ has a separable quotient if and only if there exists a family of continuous seminorms $(p_\lambda)_{\lambda\in \Lambda}$ such that $Y=\{x\in X: \sup_{\lambda\in \Lambda}p_\lambda(x)<\infty\}$ is proper and dense.
\end{theorem}

\begin{proof}
The converse follows immediately from Theorem \ref{t2.9} and Lemma \ref{t2.8}. Now, assuming that $X$ has a separable quotient, we obtain by Theorem \ref{t1} a strictly increasing sequence $Y_1\subset Y_2\subset\cdots \subset Y_n\subset\cdots$ of subspaces of $X$ whose union is dense in $X$. Consider now the sequence of seminorms $(\hat{p}_n)_{n\in\mathbb{N}}$ given by $\hat{p}_n(x)=\|\pi_{Y_n}(x)\|$. Since $\bigcup_{n\in\mathbb{N}}Y_n$ is a proper subspace of $X$ (thanks to Baire's Category Theorem), there exists $y\in X\setminus \bigcup_{n\in\mathbb{N}}Y_n$, whence we have $\hat{p}_n(y)>0$ for all $n\in\mathbb{N}$. Consider now the new sequence of continuous seminorms $(p_n)_{n\in\mathbb{N}}$ given by
\[
p_n(x)=\sum_{i=1}^{n}\frac{\hat{p}_i(x)}{\hat{p}_i(y)}.
\]
Let us now see that $\bigcup_{n\in\mathbb{N}}Y_n\subset Y=\{x\in X: \sup_{n\in \mathbb{N}}p_n(x)<\infty\}$ and $y\notin Y$ to conclude the result. In fact, given $x\in \bigcup_{n\in\mathbb{N}}Y_n$, there exists $N\in\mathbb{N}$ such that $x\in Y_N$. Thus $\hat{p}_m(x)= 0$ for all $m\geq N$, where we obtain $p_N(x)=p_n(x)$ for all $n\geq N$. Therefore $\sup_{n\in \mathbb{N}}p_n(x)<\infty$, and then $x\in Y$. Finally, since $p_n(y)=n$ for all $n\in\mathbb{N}$, we have that $y\notin Y$.
\end{proof}

As a direct consequence of Lemma \ref{t2.8}, we have the following result:

\begin{proposition}\label{c2.7}
Let $X$ be a Banach space. If $(p_n)_{n\in\mathbb{N}}\subset [1,\infty)$ is a non-decreasing sequence, and $(\alpha^{(n)})_{n\in\mathbb{N}}\subset[0,1]^{\mathbb{N}}$ and $(\varphi_n)_{n\in\mathbb{N}}\subset X^*\setminus\{0\}$ are such that for each $n\in\mathbb{N}$ holds
\begin{align*}
& \left\{x\in X: \sum_{k=1}^{\infty}\left| \left(\prod_{j=1}^{n}\alpha^{(j)}_{k}\right)\cdot\varphi_k(x)\right|^{p_n}<\infty\right\} \\
& \qquad \qquad \subsetneq\left\{x\in X: \sum_{k=1}^{\infty}\left| \left(\prod_{j=1}^{n+1}\alpha^{(j)}_{k}\right)\cdot\varphi_k(x)\right|^{p_{n+1}}<\infty\right\},
\end{align*}
then there exists a strictly increasing sequence of operator ranges $R_1\subset R_2\subset\ldots\subset R_n\subset\ldots$ with
\[
\left\{x\in X: \sum_{k=1}^{\infty}\left| \left(\prod_{j=1}^{n}\alpha^{(j)}_{k}\right)\cdot\varphi_k(x)\right|^{p_n}<\infty\right\}\subset R_n
\]
for all $n\in\mathbb{N}$.

In particular, if $X$ is a Banach space and $$Y=\left\{x\in X: \sum_{n=1}^{\infty}|\varphi_n(x)|^p<\infty\right\}$$ is proper, where $p\geq 1$ and $(\varphi_n)_{n\in\mathbb{N}}\subset X^*$ is a sequence of non-zero functionals, then for each $z\in X\setminus Y$ there exists an operator range $R_z$ such that $z\notin R_z$ and $Y\subset R_z$.
\end{proposition}

Proposition \ref{c2.7} provides a condition on the Banach space $X$ that ensures the existence of an operator range in $X$, but it does not impose any condition guaranteeing that such a set or its closure does not coincide with $X$. In view of Theorem \ref{t2.9}, if, for instance, the subspace $Y = \left\{ x \in X : \sum_{n=1}^{\infty} |\varphi_n(x)|^p < \infty \right\}$ in the particular case of Proposition \ref{c2.7} is proper and dense, then it follows that $X$ has a separable quotient. The same conclusion holds in Propositions \ref{p2.11} and \ref{p8a} below.

Proposition \ref{c2.7} was obtained before Lemma \ref{t2.8}, using ideas from the well-known characterization of weakly compactly generated spaces due to Davis, Figiel, Johnson and Pe\l{}czy'nski \cite{Davis1974}. \textcolor{black}{It was only later} that we proved Lemma \ref{t2.8} and realized that Proposition \ref{c2.7} was a particular case of it. Nevertheless, we decided to keep Proposition \ref{c2.7} in the paper because, using the technique that was initially employed to prove Proposition \ref{c2.7} (ideas from the well-known characterization of weakly compactly generated spaces \cite{Davis1974}), we were able to derive additional results --variants of Proposition \ref{c2.7}-- that are not covered by Lemma \ref{t2.8}. We present one of these results here to preserve the core idea of the proof.

\begin{proposition}\label{p2.11}
Let $X$ be a Banach space, $p\in (1,\infty]$ and $(\varphi_n)_{n\in\mathbb{N}}\subset X^*$ a sequence of non-zero functionals. If
\[
Y=\bigcup\limits_{q\in [1,p)}\left\{x\in X: \sum_{n=1}^{\infty}|\varphi_n(x)|^q<\infty\right\}
\]
is a proper subspace of $X$, then for each $z\in X\setminus Y$ there exists an operator range $R_z$ such that $z\notin R_z$ and $Y\subset R_z$.
\end{proposition}

\begin{proof}
Let $(q_n)_{n\in\mathbb{N}}$ be a strictly increasing sequence in $[1,p)$ such that $q_n\rightarrow p$. Since
\[
\left\{x\in X: \sum_{n=1}^{\infty}|\varphi_n(x)|^r<\infty\right\}\subset\left\{x\in X: \sum_{n=1}^{\infty}|\varphi_n(x)|^q<\infty\right\}
\]
for $r<q$, we have $Y=\bigcup\limits_{n\in\mathbb{N}}\left\{x\in X: \sum_{i=1}^{\infty}|\varphi_i(x)|^{q_n}<\infty\right\}$. Since $z\notin Y$, it follows that $\sum_{i=1}^{\infty}|\varphi_i(x)|^{q_n}=\infty$ for all $n\in\mathbb{N}$, and so for every $k\in\mathbb{N}$ we can take $n_k\in \mathbb{N}$ satisfying
\[
\sum_{i=1}^{n_k}|\varphi_i(z)|^{q_k}\geq (\sqrt{k}-\|z\|)^{q_k}.
\]
Define in $X$ the norm
\[
\|x\|_k=\frac{1}{k}\left(\|x\|+\left(\sum_{i=1}^{n_k}|\varphi_i(x)|^{q_k}\right)^\frac{1}{q_k}\right),
\]
which is clearly equivalent to the original norm of $X$, since
\[
\frac{1}{k}\|x\|\leq \|x\|_k\leq \frac{1+\left(\sum_{i=1}^{n_k}\|\varphi_i\|^{q_k}\right)^\frac{1}{q_k}}{k}\|x\|.
\]

If we consider $R_z=\{x\in X: \sum_{k=1}^{\infty}\|x\|_k^2<\infty\}$, we have that $R_z$ is a Banach space if equipped with the norm
\[
|||x|||=\sqrt{\sum_{k=1}^{\infty}\|x\|_k^2}
\]
and the inclusion of $R_z$ in $X$ is continuous. Given $y\in Y$, we have that $y\in\{x\in X: \sum_{n=1}^{\infty}|\varphi_n(x)|^{q_r}<\infty\}$ for some $r\in \mathbb{N}$ and $$\left(\sum_{n=1}^{\infty}|\varphi_n(y)|^{q_s}\right)^\frac{1}{q_s}\leq\left(\sum_{n=1}^{\infty}|\varphi_n(y)|^{q_r}\right)^\frac{1}{q_r}$$ for all $s\geq r$. So
\begin{align*}
&\sum_{k=1}^{\infty}\|y\|_k^2\\
&=\sum_{k=1}^{r-1}\frac{1}{k^2}\left(\|y\|+\left(\sum_{i=1}^{n_k}|\varphi_i(y)|^{q_k}\right)^\frac{1}{q _k}\right)^2+\sum_{k=r}^{\infty}\frac{1}{k^2}\left(\|y\|+\left(\sum_{i=1}^{n_k}|\varphi_i(y)|^{q_k}\right)^\frac{1}{q_k}\right)^2\\
&\leq\sum_{k=1}^{r-1}\frac{1}{k^2}\left(\|y\|+\left(\sum_{i=1}^{n_k}|\varphi_i(y)|^{q_k}\right)^\frac{1}{q_k}\right)^2 +\sum_{k=r}^{\infty}\frac{1}{k^2}\left(\|y\|+\left(\sum_{i=1}^{n_k}|\varphi_i(y)|^{q_r}\right)^\frac{1}{q_r}\right)^2\\
&\leq\sum_{k=1}^{r-1}\frac{1}{k^2}\left(\|y\|+\left(\sum_{i=1}^{n_k}|\varphi_i(y)|^{q_k}\right)^\frac{1}{q_k}\right)^2+ \left(\|y\|+\left(\sum_{i=1}^{\infty}|\varphi_i(y)|^{q_r}\right)^\frac{1}{q_r}\right)^2\cdot\left(\sum_{k=r}^{\infty}\frac{1}{k^2}\right)\\
&<\infty,
\end{align*}
that is, $y\in R_z$. On the other hand,
\begin{align*}
\sum_{k=1}^{\infty}\|z\|_k^2 & =\sum_{k=1}^{\infty}\frac{1}{k^2}\left(\|z\|+\left(\sum_{i=1}^{n_k}|\varphi_i(z)|^{q_k}\right)^\frac{1}{q_k}\right)^2 \\
& \geq \sum_{k=1}^{\infty}\frac{1}{k^2}\left(\|z\|+\sqrt{k}-\|z\|\right)^2=\sum_{k=1}^{\infty}\frac{1}{k}=\infty.
\end{align*}
\end{proof}

Using the ideas from the proof of the proposition above, we obtain a variant of Proposition \ref{c2.7}. Proposition \ref{p8a} below generalizes Proposition \ref{p2.11}.

\begin{proposition}\label{p8a}
Let $X$ be a Banach space. Let $(p_n)_{n\in\mathbb{N}}\subset [1,\infty)$ be a non-decreasing sequence, $(\alpha^{(n)})_{n\in\mathbb{N}}\subset[0,1]^{\mathbb{N}}$, and $(\varphi_n)_{n\in\mathbb{N}}\subset X^*\setminus\{0\}$. For each $n\in\mathbb{N}$, define
\[
Y_n:=\left\{x\in X: \sum_{k=1}^{\infty}\left| \left(\prod_{j=1}^{n}\alpha^{(j)}_{k}\right)\cdot\varphi_k(x)\right|^{p_n}<\infty\right\}.
\]
If the subspace $\bigcup_{n\in\mathbb{N}}Y_n$ is proper, then for each $z\in X\setminus \bigcup_{n\in\mathbb{N}}Y_n$, there exists an operator range $R_z$ such that $z\notin R_z$ and $\bigcup_{n\in\mathbb{N}}Y_n\subset R_z$.  
\end{proposition}

As previously mentioned, in order for Propositions \ref{c2.7}, \ref{p2.11}, and \ref{p8a} to provide sufficient conditions ensuring the existence of a separable quotient in the Banach space under consideration, it is necessary to assume that the operator ranges appearing in the statements of these propositions are proper and dense. Below, we present additional conditions on the Banach space $X$ that allow us to guarantee the existence of separable quotients in $X$ even when the operator range 
\[
Y = \left\{ x \in X : \sum_{n=1}^{\infty} |\varphi_n(x)|^p < \infty \right\}
\]
coincides with the entire space $X$.

\begin{theorem}
Let $(X,\|\cdot\|)$ and $(Y,\|\cdot\|_Y)$ be Banach spaces with $Y\subset\mathbb{K}^\mathbb{N}$, and $Y$ having the canonical sequence $(e_n)_{n\in\mathbb{N}}$ as Schauder basis. Let $(\varphi_n)_{n\in\mathbb{N}}$ be a linearly independent sequence in $X^*$. If $(\varphi_n(x))_{n\in\mathbb{N}}\in Y$ for all $x\in X$ and $$\sup_{(\varphi_n(x))_{n\in\mathbb{N}}\in S_Y}\left[\inf_{y\in \bigcap_{n\in\mathbb{N}}\mathrm{ker}(\varphi_n)}\|x-y\|\right]<\infty,$$ then $X$ has quotient isomorphic to a subspace of $Y$.
\end{theorem}

\begin{proof}
Since the map $T:X\longrightarrow Y$ given by $T(x)=(\varphi_n(x))_{n\in\mathbb{N}}$ is well defined and has a closed graph, it follows that there exists $C>0$ such that
\[
\|(\varphi_n(x))_{n\in\mathbb{N}}\|_Y\leq C\|x\|,
\]
for all $x\in X$. Consider now the quotient space $\displaystyle W:=\frac{X}{\bigcap\limits_{n\in\mathbb{N}}\mathrm{ker}(\varphi_n)}$. For $\overline{x}\in W$ we define $\overline{\varphi}_n(\overline{x}):=\varphi_n(x)$ and this expression does not depend on the representative of the \textcolor{black}{equivalence} class $\overline{x}$. As for $y\in \bigcap\limits_{n\in\mathbb{N}}\mathrm{ker}(\varphi_n)$ we have
\[
|\overline{\varphi}_n(\overline{x})|=|\varphi_n(x)|=|\varphi_n(x-y)|\leq\|\varphi_n\|\cdot\|x-y\|,
\]
it follows that $\overline{\varphi}_n\in W^*$ with $\|\overline{\varphi}_n\|\leq\|\varphi_n\|$. It can be easily verified that the sequence $(\overline{\varphi}_n)_{n\in\mathbb{N}}$ in $W^*$ is linearly independent with $\bigcap\limits_{n\in\mathbb{N}}\mathrm{ker}(\overline{\varphi}_n)=\{0\}$, that is, $W^*=\overline{\mathrm{span}\{\overline{\varphi}_n:\ n\in\mathbb{N}\}}^{\mathrm{w}^*}$. Clearly we also have
\[
\|(\overline{\varphi}_n(\overline{x}))_{n\in\mathbb{N}}\|_Y\leq C\|\overline{x}\|,
\]
for all $\overline{x}\in W$. Since the sequence $(\overline{\varphi}_n)_{n\in\mathbb{N}}$ separates the points of $W$, it follows that
\[
\|(\overline{\varphi}_n(\overline{x}))_{n\in\mathbb{N}}\|_Y =0\Leftrightarrow \overline{x}=0.
\]
Hence, the expression $\|\overline{x}\|_W:=\|(\overline{\varphi}_n(\overline{x}))_{n\in\mathbb{N}}\|_Y$ defines a norm on $W$. Evidently the set $\{\overline{x}\in W:\ \|\overline{x}\|_W= 1\}$ is bounded in $W$ with the quotient norm, that is, there exists $K>0$ such that, if $\|\overline{x}\|_W= 1$, then $\|\overline{x}\|\leq K$. Thus, for $\overline{x}\neq 0$ we have
\[
\left\|\frac{\overline{x}}{\|\overline{x}\|_W}\right\|\leq K,
\]
whence we infer that $\|\overline{x}\|\leq K \|\overline{x}\|_W$. Since the norms $\|\cdot\|$ and $\|\cdot\|_W$ are equivalent in $W$, it follows that $(W,\|\cdot\|_W)$ is complete. Now, since the map $\overline{T}:(W,\|\cdot\|_W)\longrightarrow Y$, $\overline{T}(\overline{x}):=(\overline{\varphi}_n(\overline{x}))_{n\in\mathbb{N}}$, is an isometry, we conclude that $(W,\|\cdot\|)$ is isomorphic to a subspace of $Y$.
\end{proof}

\bigskip

\noindent \textbf{Acknowledgments.} The authors would like to thank Prof. Per H. Enflo for fruitful conversations and remarks on this paper.

\end{document}